\documentclass[a4paper, 11pt]{article}
\usepackage{amssymb,latexsym}
\usepackage{euscript}
\usepackage{eufrak}
\usepackage{amsmath}
\usepackage{amscd}

\newcommand{\negr}[1]{\boldsymbol{#1}}

\newtheorem{defi}{\bf Definition}[section]

\newtheorem{lem}{\bf Lemma}[section]
\newtheorem{theo}{\bf Theorem}[section]
\newtheorem{cor}{\bf Corollary}[section]
\newtheorem{propo}{\bf Proposition}[section]

\newenvironment{proof}{\noindent\bf Proof. \rm}{\hfill $\mbox{\boldmath{$\square$}}$}

\def\LL{\mbox{{\sl I}}\!\mbox{{\sl L}}}

\def\fun{\longrightarrow}

\title{\bf  Principal and Boolean congruences on $\negr{\theta}$-valued {\L}ukasiewicz--Moisil algebras}

\author{A. V. Figallo, I. Pascual y A. Ziliani\\[3mm]
\small Instituto de Ciencias B\'asicas, Universidad Nacional de San Juan,\\
\small 5400 San Juan, Argentina.\\
\small Departamento de Matem\'atica, Universidad Nacional del Sur,\\
\small 8000 Bah\'{\i}a Blanca, Argentina. 
}

\date{}

\begin{document}

\maketitle

\thispagestyle{empty}

\begin{abstract}

The first system of many--valued logic was introduced by J. {\L}u\-ka\-sie\-wicz, his motivation was of philosophical nature as he was looking for an interpretation of the concepts of possibility and necessity. Since then, plenty of research has been developed in this area. In 1968, when Gr.C. Moisil came across Zadeh's fuzzy set theory he found the  motivation he had been looking for in order to legitimate the introduction and study of infinitely--valued {\L}ukasiewicz algebras, so he defined $\theta-$valued {\L}ukasiewicz algebras (or  ${\rm LM}_{\theta}-$algebras, for short) (without negation), where $\theta$ is the order type of a chain. 

In this article, our main interest is to investigate the principal and Boolean congruences and ${\theta}-$congruences on  ${\rm LM}_{\theta}-$al\-ge\-bras. In order to do this we take into account a topological duality for these algebras obtained in (A.V. Figallo, I. Pascual, A. Ziliani,  {\em  A Dua\-li\-ty for $\theta-$Valued {\L}ukasiewicz--Moisil Algebras and Aplicattions}. J. LMult-Valued Logic $\&$ Soft Computing. Vol. 16, pp 303-322. (2010). Furthermore, we prove that the intersection of two principal ${\theta}-$con\-gruen\-ces is a principal one. On the other hand, we show that Boolean congruences are both principal congruences and  ${\theta}-$con\-gruen\-ces. This allowed us to establish necessary and sufficient conditions so that a principal congruence is a Boolean one. Finally, bearing in mind the above results, we characterize the principal and Boolean congruences on $n-$valued {\L}ukasiewicz algebras without negation when we consider them as ${\rm LM}_{\theta}-$algebras in the case that $\theta$ is an integer $n$, $n \geq 2$. 

\


\end{abstract}

\section{\textbf \large Introduction}\label{s01}\indent 

In 1940, Gr.C. Moisil defined the $3-$valued and the $4-$valued {\L}u\-ka\-sie\-wicz algebras and in 1942, the $n-$valued {\L}ukasiewicz algebras ($n\geq 2$). His goal was to study {\L}ukasiewicz's logic from the algebraic point of view. 
It is well-known that these algebras are not the algebraic counterpart of $n-$valued  {\L}ukasiewicz propositional calculi for $n\geq 5$ (\cite{BOI,RC1}). R. Cignoli (\cite{RC3}) found algebraic counterparts
for $n\geq 5$ and he called them proper $n-$valued {\L}ukasiewicz algebras.  On the other hand, in 1968, Gr.C. Moisil (\cite{GrM2}) introduced $\theta-$valued {\L}ukasie\-wicz algebras (without negation), where $\theta$ is the order type of a chain. These structures were thought by Moisil as models of a logic with infinity nuances, but the need to find a strong motivation for them delayed the announcement. The motivation was found when Moisil came across Zadeth's fuzzy set theory, in which he saw a confirmation of his old ideas.

\vspace{3mm}

For a general account of the origins and the theory of {\L}ukasiewicz many valued logics and {\L}ukasiewicz algebras the reader is referred to \cite{BOI,COM2,AI}.

\vspace{3mm}

This paper is organized as follows: in Section 2, we summarize the principal notions and results of $\theta-$valued {\L}uka\-sie\-wicz algebras (\cite{BOI}), in par\-ti\-cu\-lar the topological duality for these algebras obtained in  \cite{FPZ2}. 
In  Section  3, we characterize  the open subsets of the dual space associated with an ${\rm LM}_{\theta}-$algebra which determine both the principal ${\rm LM}_{\theta}$ and $\theta{\rm LM}_{\theta}-$con\-gruen\-ces. This last result enables us to  prove  that the intersection of two principal $\theta{\rm LM}_{\theta}-$con\-gruen\-ces is a principal one. Furthermore, whenever  $\theta$ is an  integer $n$, $n \geq 2$, we obtain the filters which  determine  principal congruences on $n-$valued {\L}ukasiewicz algebras (or ${\rm LM}_{n}-$algebras) and, we are also in a position to show that the intersection of two principal ${\rm LM}_{n}-$con\-gruen\-ces is a  principal one.
In  Section 4, attention is focused on Boolean  congruences. Firstly, we characterize them by means of certain closed and open subsets of the associated space. These results allow us to prove that the  boolean ${\rm LM}_{\theta}$ and $\theta{\rm LM}_{\theta}-$con\-gruen\-ces coincide, and also that they are principal congruences associated with filters generated by Boolean  elements of these algebras.

\section{\textbf \large Preliminaries}\label{s02}\indent 

In this paper, we take for granted the concepts and results on distributive lattices, universal algebra, {\L}uka\-sie\-wicz algebras  and Priestley duality. To obtain more information on these topics, we direct the reader to the bibliography indicated in  \cite{BD,CBOI,BOI, SB.HS, HP1, HP2, HP3}. However, in order to simplify the reading, we will summarize  the main notions and  results  we needed throughout this paper.

\vspace{2mm}

In what follows, if $X$ is a partially ordered set and $Y \subseteq X$ we will denote by $[Y)$ ($(Y]$) the set of all $x\in X$ such that $y\leq x$ ($x\leq y$) for some $y\in Y$, and we will say that $Y$ is increasing (decreasing) if $Y= [Y)$ ($Y=(Y]$). In particular, we will write  $[y)$ ($(y]$) instead of $[\{y\})$  ($(\{y\}]$). Furthermore, if $x,y \in X$ and $x\leq y$, a segment is the set $\{z\in X: x\leq z\leq y\}$ which will be denoted by $[x,y]$.  

\vspace{2mm}

Let $\theta \geq 2$ be the order type of a totally ordered set $J$ with least element $0$ being   $J = \{0\}+ I$ (ordinal sum). Following V. Boisescu et al (\cite{BOI}) recall that:

\vspace{2mm}

A $\theta-$valued {\L}ukasiewicz--Moisil algebra (or LM$_{\theta}-$algebra) is an algebra \linebreak  $\langle A,\vee,\wedge,0,1,\{\phi_i\}_{i\in I},\{\overline {\phi}_i\}_{i\in I}\rangle$ of   type $(2,2,0,0,\{1\}_{i\in I},\{1\}_{i\in I})$ where $\langle A,\vee,\wedge,$ $0,1 \rangle$ is a bounded distributive lattice and for all $i \in I$, $\phi_{i}$ and $\overline{\phi}_i$ satisfy the following conditions:

\begin{itemize}
\item[{\rm(L1)}] $\phi_i$ is an endomorphism of bounded distributive lattices,
\item[{\rm (L2)}] $\phi_i x \vee \overline{\phi_i}x=1$,\, $\phi_ix \wedge\overline{\phi_i}x= 0$,
\item[{\rm(L3)}] $\phi_i \phi_j x= \phi_j x$,
\item[{\rm(L4)}] $i\leqslant j$ implies $\phi_i x \leqslant \phi_j x$, 
\item[{\rm(L5)}] $\phi_ix=\phi_iy$  for all $i\in I$ imply $x=y$.
\end{itemize}

\vspace{2mm}

It is well known that there are LM$_{\theta}-$congruences (or congruences) on  LM$_{\theta}-$al\-ge\-bras such that the quotient algebra doesn't satisfy the determination principle (L5). That is the reason why a new notion was defined as follows: a $\theta$LM$_{\theta}-$congruence (or ${\theta}-$congruence) on an  LM$_{\theta}-$algebra is  a bounded distributive lattice congruence $\vartheta$ such that $(x,y)\in \vartheta$ if and only if  $(\phi_ix, \phi_iy)\in \vartheta$  for all $i\in I$. (\cite{CBOI,AI}). 

\vspace{2mm}

The following characterization of the  Boolean elements of an  LM$_{\theta}-$algebra will be useful for the study of  the Boolean congruences on these algebras: 

\begin{itemize}
\item[{\rm (L6)}] Let $A$  be  an  LM$_{\theta}-$algebra and let $C(A)$ be the set of all Boolean elements of $A$. Then, for each $x \in A$, the following conditions are equivalent:

\begin{itemize}
\item[{\rm (i)}] $x \in C(A)$,
\item[{\rm (ii)}] there are $y \in A$ and $i \in I$ such that $x = \phi_i y$ ($x = \overline{\phi}_i y$), 
\item[{\rm (iii)}]there is  $i_0 \in I$ such that $x = \phi_{i_0} x$ ($x = \overline{\phi}_{i_0} x$),
\item[{\rm (iv)}] for all $i \in I$, $x = \phi_i x$  ($x = \overline{\phi}_i x$).
\end{itemize}
\end{itemize}

\vspace{2mm}

In \cite{FPZ2}, we extended Priestley duality to LM$_{\theta}-$algebras considering $\theta-$va\-lued {\L}ukasiewicz--Moisil spaces (or l$_{\theta}$P-spaces) and $l_{\theta}P-$functions. More pre\-ci\-se\-ly, 

\vspace{2mm}

A\, $\theta-$valued {\L}ukasiewicz--Moisil space {\rm(}or $l_{\theta}P-$space{\rm )} is a pair $(X,\{f_i\}_{i\in I})$ provided the following conditions are satisfied:

\begin{itemize}
\item[{\rm(lP1)}] $X$ is a Priestley space, {\rm (}{\rm \cite{HP1,HP2,HP3})}
\item[{\rm(lP2)}] $f_i: X \to X $ is a continuous function,
\item[{\rm(lP3)}] $x \leqslant y$ implies $f_i(x)= f_i(y)$, 
\item[{\rm(lP4)}] $i \leqslant j$ implies $f_i(x)\leqslant f_j(x)$,
\item[{\rm(lP5)}] $f_i \circ f_j=f_i$.
\item[{\rm(lP6)}]$\bigcup\limits_{i\in I}f_i(X)$ is dense in $X$ {\rm(}i.e. $\overline {\bigcup\limits_{i\in I}f_i(X)} = X$, where  $\overline Z$ denotes the closure of $Z${\rm )}.
\end{itemize}

\vspace{2mm}

An\, $l_{\theta}P-$ function from an $l_{\theta}P-$ space $(X,\{f_i\}_{i\in I})$ into another $(X^{\prime},\{f_i^{\prime}\}_{i\in I})$ is an increasing continuous function $f$ from  $X$ into $X^{\prime}$ satisfying $f_i^{\prime}\circ f = f\circ f_i$ for all $i\in I$. 

\vspace{2mm}

It is worth mentioning that condition (lP6) is equivalent to the following one:
\begin{itemize}
\item[{\rm(lP7)}] If $U$ and $V$ are increasing closed and open subsets of $X$ and $f_i^{-1}(U)= f_i^{-1}(V)$ for all  $i\in I$, then $U=V$ (\cite{FPZ2}).
\end{itemize}

Besides, if $(X,\{f_i\}_{i\in I})$ is an  $l_{\theta}P-$space, then for all $x \in X$, the following properties are satisfied (\cite{FPZ2}):

\begin{itemize}
\item[{\rm(lP8)}] $x \leqslant f_i(x)$ or $f_{i}(x) \leqslant x$ for all $i\in I $, 
\item[{\rm(lP9)}] $f_0(x)\leq x$  and   $f_0(x)$ is the unique minimal element in $X$ that precedes $x$,
\item[{\rm(lP10)}] $x \leq f_1(x)$  and  $f_1(x)$ is the unique  maximal element in $X$ that follows  $x$.

Furthermore, the above properties allow us to assert that

\item[(lP11)] $X$ is the cardinal sum of the sets  $[\{f_i(x)\}_{i \in I})\cup (\{f_i(x)\}_{i \in I}]$ for $x \in X$. If $I$ has least element $0$ and  greatest element $1$,  then $X$ is the cardinal sum of the sets $[f_0(x), f_1(x)]$ for $x \in X$. 
\end{itemize}

\vspace{2mm}

Although in \cite{FPZ2} we developed a topological duality for LM$_{\theta}-$algebras, next we will describe  some  results of it with the aim of fixing the notation we are about to use in this paper.  

\vspace{2mm}

\begin{itemize}
\item[(A1)] If $(X,\{f_i\}_{i\in I})$ is an $l_{\theta}P-$space and $D(X)$ is the lattice of all increasing closed and open subset of $X$, then ${\LL}_{\theta}= \langle D(X),\cup,\cap, \emptyset, X, \{\phi_i^X \}_{i \in I},\\ \{\overline{\phi}_i^{X}\}_{i \in I}\rangle$ is an  ${\rm LM}_{\theta}-$algebra, where the operations $\phi_i^X$ and $\overline{\phi}_i^{X}$ are defined by means of the formulas: $\phi_i^{X}(U)= {f_i}^{-1}(U)$ and $\overline{\phi}_i^{X}(U)= X \setminus 
 {f_i}^{-1}(U)$ for all $U \in D(X)$ and for all $i \in I$. 

\item[(A2)] If $\langle A,\vee,\wedge,0,1,\{\phi_i\}_{i\in I},\{\overline {\phi}_i\}_{i\in I}\rangle$ is an ${\rm LM}_{\theta}-$algebra and $X(A)$ is the set of all prime filters of $A$, ordered by inclusion  and with the to\-po\-lo\-gy having as a sub--basis the sets $\sigma_{A}(a)= \{ P \in X(A):  a \in P \}$ and \, $X(A)\setminus \sigma_{A}(a)$ for each $a \in A$, then ${\L}_{\theta}= (X(A),\{ f_i^{A}\}_{i \in I})$ is the $l_{\theta}P-$space associated with $A$, where the functions $f_i^{A}: X(A) \fun X(A)$ are defined by the prescription: $f_i^{A}(P)= {\phi_i}^{-1}(P)$ for all $i \in I$ and for all $P \in X(A)$. 
\end{itemize}
Then the category of $l_{\theta}P-$spaces and  $l_{\theta}P-$func\-tions  is naturally equivalent to the dual of the category of LM$_{\theta}-$algebras and their corresponding homomorphism, where the isomorphisms $\sigma_{A}$ and $\epsilon_X$ are the corresponding natural equivalences.

\vspace{2mm}

In addition, this duality allowed us to to characterize the ${\rm LM}_{\theta}$  and  $\theta{\rm LM}_{\theta}-$ congruences on these algebras for which we introduced these notions: 

\begin{itemize}
\item[(A3)] A subset $Y$ of $X$ is semimodal if $\bigcup\limits_{i \in I}f_i(Y)\subseteq  Y$, and $Y$ is a $\theta-$subset of $X$ if $\bigcup\limits_{i \in I}f_i(Y)\subseteq  Y \subseteq  \overline{\bigcup\limits_{i \in I}f_i(Y)}$. 
Hence, for each closed $\theta-$subset $Y$ of $X$ we have that $Y =\overline{\bigcup\limits_{i \in I}f_i(Y)}$ (\cite{FPZ2}). 
\end{itemize}

Then, we proved that

\begin{itemize}
\item[(A4)]  The lattice ${\cal C}_{S}({\L}_{\theta}(A))$ of all closed and semimodal subsets of ${\L}_{\theta}(A)$ is isomorphic to the dual lattice $Con_{LM_{\theta}}(A)$ of all congruences on $A$,  and the isomorphism is the function 
$\Theta_{S}$ defined by the prescription 

\item[]$\Theta_S (Y)=\{(a,b)\in  A\times A: \sigma_{A}(a)\cap Y= \sigma_{A}(b) \cap Y\}$
\item[]\hspace{1.1cm} $=\{(a,b) \in  A\times A:  \sigma_{A}(b)\triangle \sigma_{A}(a) \cap Y = \emptyset \}$ (\cite[Theorem 2.1.1]{FPZ2}).

\item[(A5)] The lattice ${\cal C}_{\theta}({\L}_{\theta}(A))$ of all closed $\theta-$subsets of ${\L}_{\theta}(A)$ is isomorphic to the dual lattice  $Con_{\theta LM_{\theta}}(A)$ of all ${\theta}-$congruences on $A$, and the isomorphism is the function $\Theta_{\theta}$ defined as in (A4) (\cite[Theorem 2.1.2]{FPZ2}).
\end{itemize}

\vspace{2mm}

Finally, we will emphasize the following properties of Priestley spaces and so, $l_{\theta}P-$spaces which will be quite useful in order to characterize the principal congruences on {\rm LM$_\theta-$}algebras. 

\begin{itemize}
\item[(A6)] $[A)$ and $(A]$ are closed subsets of $X$, for each closed subset $A$ of $X$. 
\vspace{3mm}

\item[(A7)] $R$ is a closed, open and  convex subset of $X$ if and only if there are $U, V \in D(X)$ such that $V \subseteq U$ and $R = U \setminus V$.
\end{itemize}

\section{\textbf Principal congruences}\label{s03}\indent 

In this section our first objective is to characterize the principal ${\rm LM}_{\theta}$ and $\theta{\rm LM}_{\theta}-$congruences on an ${\rm LM}_{\theta}-$algebra by means of certain open subsets of its associated ${\rm l}_{\theta}$P--space.

\vspace{2mm}

\begin{theo}\label{T21}
Let  $A$ be an ${\rm LM}_{\theta}-$algebra and let ${\L}_{\theta}(A)$ be the ${\rm l}_{\theta}$P--space associated with $A$. Then, it holds:

\begin{itemize}
\item[{\rm(i)}] the lattice ${\cal O}_{CS}(X(A))$ of all open subsets of $X(A)$ whose complements are semimodal is  isomorphic to the lattice $Con_{LM_{\theta}}(A)$ of all congruences on $A$, and the isomorphism is the function $\Theta_{OS}$  defined by the prescription  $\Theta_{OS}(G)=\{(a,b)\in A\times A: \sigma_A(b) \triangle \sigma_A(a)\subseteq 
 G\}$.    
  
\item[{\rm(ii)}] The lattice ${\cal O}_{C\theta}(X(A))$ of all open subsets of $X(A)$ whose complements are $\theta-$subsets of $X(A)$ is  isomorphic to the lattice $Con_{\theta LM_{\theta}}(A)$ of all ${\theta}-$con\-gruences on $A$  and the isomorphism is the function $\Theta_{O\theta}$ defined as in {\rm(i)}.  
\end{itemize}
\end{theo}

\begin{proof}
It is a direct consequence of (A4) and (A5), bearing in mind that there is a one-to-one correspondence between the closed and open subsets of a topological space and that $\Theta_{OS}(G)= \Theta_{S}(X(A) \setminus G)$ and  $\Theta_{O\theta}(G) = \Theta_{\theta}(X(A) \setminus G)$.   
\end{proof}

\

From now on, we will denote by $\Theta(a,b)$ and $\Theta_\theta (a,b)$ the principal  ${\rm LM}_{\theta}$ and $\theta{\rm LM}_{\theta}-$congruence generated by $(a,b)$, respectively.

\vspace{2mm}

\begin{propo}\label{C21}
Let  $A$ be an  ${\rm LM}_{\theta}-$algebra and let ${\L}_{\theta}(A)$ be the ${\rm l}_{\theta}$P--space associated with $A$. Then the following conditions are equivalent for all $a,b \in A$, $a \leqslant b$:  
 
\begin{itemize}
\item[\rm{(i)}]  $\Theta(a,b)= \Theta_{OS}(G)$ for some $G \in {\cal O}_{CS}(X(A))$,
 \item[ \rm{(ii)}] $G$ is the least element of ${\cal O}_{CS}(X(A))$,  ordered by inclusion, which contains $\sigma_A(b) \setminus \sigma_A(a)$,
\item[\rm{(iii)}] $G = (\sigma_A(b) \setminus \sigma_A(a))\cup \bigcup\limits_{\scriptstyle i \in I}{{f_i}^A}^{-1}(\sigma_A(b) \setminus \sigma_A(a))$,
\item[\rm{(iv)}] $G =(\sigma_A(b) \setminus \sigma_A(a))\cup \bigcup\limits_{\scriptstyle i \in I}\sigma_A(\phi_i b  \wedge \overline{\phi}_i a) $.
\end{itemize}
\end{propo}

\begin{proof} 
(i) \, $\Rightarrow $ \, (ii): By Theorem \ref{T21} we have that $\sigma_A(b)\setminus\sigma_A(a)\subseteq  G$. On the other hand, suppose that there is  $H \in {\cal O}_{CS}(X(A))$  such that $\sigma_A(b)\setminus \sigma_A(a)\subseteq  H$. Then, $(a,b)\in \Theta_{OS}(H)$ and so, $\Theta_{OS}(G) \subseteq \Theta_{OS}(H)$. This assertion and Theorem \ref{T21} imply that $G \subseteq H$.

\vspace{2mm}

(ii)\, $\Rightarrow$\, (i): From the  hypothesis and  Theorem \ref{T21} we conclude that $(a,b)\in \Theta_{OS}(G)$. Furthermore,  if $ \varphi \in Con_{LM_{\theta}}(A) $ and  $(a, b)\in \varphi$, then by  Theorem \ref{T21} we have that  $\varphi=\Theta_{OS}(H)$ for some $H \in {\cal O}_{CS}(X(A))$. From these last assertions and the fact that $a \leq b$ we infer that $\sigma_A(b)\setminus \sigma_A(a)\subseteq  H$ and so, by the hypothesis we conclude that $G \subseteq H$. This means that $\Theta_{OS}(G) \subseteq \Theta_{OS}(H)=\varphi$, which allows us to assert that $\Theta_{OS}(G)=\Theta(a,b)$.

\vspace{2mm}

(ii)\, $\Rightarrow$\, (iii): From the  hypothesis we have that $(\sigma_A(b) \setminus \sigma_A(a))\cup \bigcup\limits_{ \scriptstyle i \in I}{{f_i}^A}^{-1}$ $(\sigma_A(b) \setminus \sigma_A(a) )\subseteq G$. Besides, from (lP2) and (lP5)  we infer that  $(\sigma_A(b) \setminus \sigma_A(a))\cup \bigcup\limits_{\scriptstyle i \in I}{{f_i}^A}^{-1}(\sigma_A(b) \setminus \sigma_A(a))\in {\cal O}_{CS}(X(A))$ and hence, by (ii) we conclude that $G = (\sigma_A(b) \setminus \sigma_A(a)) \cup \bigcup\limits_{\scriptstyle i \in I}{{f_i}^A }^{-1}(\sigma_A(b) \setminus \sigma_A(a))$.

\vspace{2mm}

(iii)\, $\Rightarrow$\, (ii): From (lP2) and (lP5) we have that $G\in {\cal O}_{CS}(X(A))$. Suppose now, that there is $H \in{\cal O}_{CS}(X(A))$ and that $\sigma_A(b)\setminus \sigma_A(a)\subseteq H$. Since $X(A)\setminus H$ is semimodal, we infer that ${f^A_i }^{-1}(\sigma_A(b) \setminus \sigma_A(a))\subseteq H$ for all $i \in I$ from which it follows that
 $G \subseteq H$.  

\vspace{2mm}

(iii)\, $\Leftrightarrow$\, (iv): It is a direct consequence from the fact that $\sigma_A$ is an isomorphism.
\end{proof}

\begin{cor}\label{C23}
Let  $A$ be an  ${\rm LM}_{\theta}-$algebra and let ${\L}_{\theta}(A)$ be the ${\rm l}_{\theta}$P--space associated with $A$. Then the following conditions are equivalent: 

\begin{itemize}
\item[\rm{(i)}] $G \in {\cal O}_{CS}(X(A))$ and $\Theta_{OS}(G)$ is a principal congruence on $A$,
\item[\rm{(ii)}] there is a closed, open and convex subset $R$ of $X(A)$ such that  $G = R \cup\bigcup\limits_{\scriptstyle i \in I}{f_i^A }^{-1}(R)$.
\end{itemize}
\end{cor}

\begin{proof} 
It follows immediately from (A7) and Proposition \ref{C21} taking into account that $\sigma_A$ is an ${\rm LM}_{\theta}-$isomorphism.
\end{proof}

\

Next, bearing in mind the above results we will obtain different descriptions of the elements of ${\cal O}_{C\theta}(X(A))$ by means of the duality which will be useful later on.

\begin{propo}\label{P25} 
Let  $A$ be an  ${\rm LM}_{\theta}-$algebra and let ${\L}_{\theta}(A)$ be the ${\rm l}_{\theta}$P--space associated with $A$. Then the following conditions are equivalent for all $a,b \in A$, $a \leqslant b$:  
  
\begin{itemize}
\item[\rm{(i)}] $H\in {\cal O}_{C\theta}(X(A))$ and $\Theta_{O\theta}(H) = \Theta_\theta (a,b)$,

\item[\rm{(ii)}] $H = X(A)\setminus \overline{\bigcup\limits_{\scriptstyle i \in I}f^A_i(X(A)\setminus G)}$, 
 where  $G =(\sigma_A(b)\setminus \sigma_A(a)) \cup \bigcup\limits_{\scriptstyle i \in I}{f^A_i}^{-1}(\sigma_A(b) \setminus \sigma_A(a))$, 

\item[\rm{(iii)}] $H$ is the least element of ${\cal O}_{C\theta}(X(A)$, ordered by inclusion, which contains $\sigma_A(b) \setminus \sigma_A(a)$,

\item[\rm{(iv)}] $H = X(A) \setminus \overline {\bigcup\limits_{\scriptstyle i \in I} f^A_i(\bigcap\limits_{\scriptstyle i \in I}{f^A_i}^{-1}(X(A )\setminus (\sigma_A(b) \setminus \sigma_A(a))))}$,

\item[\rm{(v)}] $H = X(A) \setminus \overline {\bigcup\limits_{\scriptstyle i \in I} f^A_i (\bigcap\limits_{\scriptstyle i \in I} \sigma_A (\phi_i a \vee \overline{\phi}_ib))}$.
\end{itemize}
\end{propo}

\begin{proof} 
(i) \, $\Rightarrow $ \, (ii): 
Let us observe that $\Theta (a,b) \subseteq \Theta_\theta (a,b)$. Hence, by Theorem \ref{T21} we have that there are $G \in {\cal O}_{CS}(X(A))$ and  $H \in {\cal O}_{C\theta}(X(A)$ such that $\Theta (a,b) = \Theta_{OS}(G)$,  $\Theta_\theta (a,b)= \Theta_{O\theta}(H)$ and  $G \subseteq H$. Then, by  Proposition \ref{C21} we infer that  $G =(\sigma_A(b) \setminus \sigma_A(a))\cup \bigcup\limits_{\scriptstyle i \in I}{f^A_i}^{-1}(\sigma_A(b) \setminus \sigma_A(a))$. Besides, it holds that $\Theta_\theta (a,b)$ is the least $\theta-$congruence on $A$ which contains  $\Theta (a,b)$. Hence, from these statements and Theorem \ref{T21} we have that  $X(A) \setminus H$ is the greatest closed $\theta-$subset of $X(A)\setminus G$. On the other hand, taking into account that $X(A)\setminus G$ is semimodal, we conclude that
$\overline {\bigcup\limits_{\scriptstyle i \in I} f^A_i ( X(A)\setminus G)}$ is the greatest closed $\theta-$subset of $X(A)\setminus G$. Hence, from the last assertions we conclude that $X(A) \setminus H = \overline {\bigcup\limits_{\scriptstyle i \in I} f^A_i ( X(A)\setminus G)}$ and so, $H = X(A) \setminus \overline {\bigcup\limits_{\scriptstyle i \in I} f^A_i ( X(A)\setminus G)}$.

\vspace{2mm}

(ii)\, $\Rightarrow$\, (iii): It is easy to check that $H\in {\cal O}_{C\theta}(X(A)$. Furthermore, since $X(A)\setminus G$ is semimodal, we have that $\overline {\bigcup\limits_{\scriptstyle i \in I} f^A_i (X(A)\setminus G)}\subseteq X(A)\setminus G$ and so, $G \subseteq H$. This condition and the hypothesis imply that $\sigma_A(b)\setminus \sigma_A(a)\subseteq H$. On the other hand, if $W \in {\cal O}_{C\theta}(X(A)$ and $\sigma_A(b)\setminus \sigma_A(a)\subseteq W$, then $H\subseteq W$. Indeed, since $X(A)\setminus W$ is semimodal, from Proposition \ref{C21} we have that $G \subseteq W$ which implies that $\overline {\bigcup\limits_{\scriptstyle i \in I} f^A_i (X(A)\setminus W)} \subseteq  \overline {\bigcup\limits_{\scriptstyle i \in I} f^A_i (X(A)\setminus G)}$. This last statement and the fact that $X(A)\setminus W$ is a closed $\theta-$subset of $X(A)$ we infer that $X(A)\setminus W \subseteq \overline {\bigcup\limits_{ \scriptstyle i \in I} f^A_i (X(A)\setminus G)}$ which completes the proof.  

\vspace{2mm}

(iii)\, $\Rightarrow$\, (i): This follows using an analogous reasoning to the proof of (ii) $\Rightarrow$ (i) in Proposition \ref{C21}. 

\vspace{2mm}

(ii)\, $\Leftrightarrow$\, (iv): We only prove that 
$f^A_i((X(A)\setminus (\sigma_A(b) \setminus \sigma_A(a))) \cap \bigcap\limits_{\scriptstyle i \in I} {f^A_i}^{-1}\\  (X(A) \setminus (\sigma_A(b)\setminus \sigma_A(a))))= f^A_i (\bigcap\limits_{\scriptstyle i \in I}{f^A_i}^{-1}(X (A)\setminus (\sigma_A(b)\setminus \sigma_A(a) )))$ for all  $i \in  I$, from which the proof follows immediately. 
Indeed, suppose that  $y = f_j(x)$ for some $j\in I$ and that $f_i(x)\not\in \sigma_A(b)\setminus \sigma_A(a)$  for all $i \in I$. Then, by (lP5) we infer that  $f_i(y) = f_i(x)$ and $f_i(y)\not\in \sigma_A(b)\setminus \sigma_A(a)$  for all $i \in I$. From these statements we have that $y \in (X(A)\setminus (\sigma_A(b) \setminus \sigma_A(a)))\cap \bigcap\limits_{  \scriptstyle i \in I} {f^A_i}^{-1}(X (A) \setminus (\sigma_A(b) \setminus \sigma_A(a))$ and so, $y \in f^A_j ((X(A)\setminus (\sigma_A(b)\setminus \sigma_A(a))) \cap \bigcap\limits_{\scriptstyle i \in I}{f^A_i}^{-1} ( X (A)\setminus (  \sigma_A(b) \setminus \sigma_A(a))))$. Therefore, we conclude that $f^A_j (\bigcap\limits_{\scriptstyle i \in I}{f^A_i}^{-1} (X (A)\setminus (\sigma_A(b) \setminus \sigma_A(a) )))$  $\subseteq  f^A_j ( (X(A)\setminus (\sigma_A(b) \setminus \sigma_A(a))) \cap \bigcap\limits_{\scriptstyle i \in I} {f^A_i}^{-1} (X (A)\setminus (\sigma_A(b) \setminus \sigma_A(a))))$ for all $j \in I$. The other inclusion is obvious.  

\vspace{2mm}

(iv)\, $\Leftrightarrow$\, (v): It is routine.  
\end{proof}

\begin{cor}\label{P26} 
Let  $A$ be an  ${\rm LM}_{\theta}-$algebra and let ${\L}_{\theta}(A)$ be the ${\rm l}_{\theta}$P--space associated with $A$. Then the following conditions are equivalent:

\begin{itemize}
\item[\rm{(i)}]  $\Theta_{O\theta}(H)$ is a principal ${\theta}-$congruence on $A$, 
\item[\rm{(ii)}] there is a closed, open and convex subset  $R$ of $X(A)$ such that
\item[] $H = X(A)\setminus \overline {\bigcup\limits_{\scriptstyle i \in I} f^A_i(\bigcap\limits_{\scriptstyle i \in I } {f^A _i}^{-1}(X(A) \setminus R))}$.
\end{itemize}
\end{cor}

\begin{proof} 
It is an immediate consequence of Proposition \ref{P25} and (A7).
\end{proof}

\begin{cor}\label{C36} 
Let  $A$ be an  ${\rm LM}_{\theta}-$algebra. Then the intersection of two principal $\theta-$congruences is a principal one. 
\end{cor}

\begin{proof}
Let $\vartheta_1$ and $\vartheta_2$ be principal $\theta-$congruences on $A$. Then, by Corollary \ref{P26}, 
there are closed, open and convex subsets $R_1$ and $R_2$ of $X(A)$ such that $ \vartheta_1 = \Theta_{O\theta}(H_1)$ and 
$\vartheta_2=\Theta_{O\theta}(H_2)$, where $H_1 = X(A)\setminus \overline {\bigcup\limits_{\scriptstyle i \in I} f^A_i(  \bigcap\limits_{\scriptstyle i \in I} {f^A _i }^{-1}( X(A) \setminus R_1))}$ and  $H_2 = X(A)\setminus\overline {\bigcup\limits_{\scriptstyle i \in I} f^A_i (\bigcap\limits_{\scriptstyle i \in I} {f^A _i }^{-1}(X(A) \setminus R_2))}$.
Bearing in mind Theorem \ref{T21} we infer that $\vartheta_1\cap\vartheta_2 = \Theta_{O\theta}(H_1 \cap H_2)$.
On the other hand, we have that $H_1 \cap H_2 = X(A)\setminus \overline{\bigcup\limits_{\scriptstyle i \in I} f^A_i(\bigcap\limits_{\scriptstyle j \in I} {f^A _j }^{-1}(X(A)\setminus R_1)}$ $\overline{\cup \bigcap\limits_{\scriptstyle k \in I} {f^A _k }^{-1}(X(A) \setminus R_2))}= X(A)\setminus  \overline {\bigcup\limits_{\scriptstyle i \in I} f^A_i (\bigcap\limits_{\scriptstyle i \in I}({f^A _i }^{-1}( X(A)\setminus R_1) \cup {f^A _i }^{-1}( X(A)}$ $\overline{\setminus R_2)))}= X(A)\setminus \overline {\bigcup\limits_{\scriptstyle i \in I} f^A_i (\bigcap\limits_{\scriptstyle i \in I}({f^A _i}^{-1}(X(A) \setminus (R_1 \cap R_2)))}$. From these last equalities  and the fact that $R_1\cap R_2$ is a closed, open and convex subset of $X(A)$ we conclude, by Corollary \ref{P26}, that $\vartheta_1\cap\vartheta_2$ is a principal $\theta-$congruence on  $A$. 
 \end{proof}

\

In the sequel, we will determine sufficient conditions for the intersection of two principal congruences  is not  a principal one, in the particular case that the  ${\rm l}_{\theta}$P--space associated with an ${\rm LM}_{\theta}-$algebra is  the cardinal sum  of an arbitrary but not finite set of segments or  by \cite[Corollary 2.1.5, Theorem 2.2.2]{FPZ2} when  this is  isomorphic  to a subdirect  product of an arbitrary  but not finite set of subalgebras  of the ${\rm LM}_{\theta}-$al\-ge\-bra ${B_2}^{[I]}$.

\vspace{2mm}

\begin{propo}\label{P317} 
Let $\varphi_1=\Theta_{OS}(G_1)$ and $\varphi_2=\Theta_{OS}(G_2)$  principal congruences on $A$ where $G_j= R_j\cup \bigcup\limits_{\scriptstyle i \in I} {f^A_i}^{-1}(R_j)$ and $R_j$ is a closed, open and convex subset of $X(A)$, $1\leq j \leq 2$. If $(R_1 \setminus (R_2  \cup  \bigcup\limits_{\scriptstyle i \in I}{f^A_i }^{-1}(R_1)))\cap \bigcup\limits_{\scriptstyle i \in I}{f ^A_ i}^{-1}(R_2)$ is a proper and dense subset of $R_1 \setminus (R_2\cup \bigcup\limits_{\scriptstyle i \in I}{f^ A_i }^{-1}(R_1))$, then $\varphi_1 \cap \varphi_2$ is not a principal congruence  on  $A$.
\end{propo}

\begin{proof}
By virtue of  Theorem \ref{T21} we have that (1)\, $\varphi_1 \cap \varphi_2 = \Theta_{OS}(G_1 \cap G_2)$. Besides, from the  hypothesis we infer that $G_1 \cap G_2$ is partitioned into mutually disjoint sets as follows (2) $G_1\cap G_2 = (R_1 \cap R_2) \cup (((R_1\setminus R_2)\setminus \bigcup\limits_{\scriptstyle i \in I} {f^A_i}^{-1}(R_1))\\ \cap \bigcup\limits_{\scriptstyle i \in I} {f^A_i}^{-1}(R_2))\cup (((R_2\setminus R_1)\setminus \bigcup\limits_{\scriptstyle i \in I} {f^A_i }^{-1}(R_2))\cap \bigcup\limits_{\scriptstyle i \in I}{f^A_i }^{-1}(R_1))\cup ( \bigcup\limits_{\scriptstyle i \in I} {f^A_i}^{-1}(R_1) \cap \bigcup\limits_{\scriptstyle i \in I} {f^A_i}^{-1}(R_2))$.  Suppose now, that $\varphi_1 \cap \varphi_2$ is a principal congruence on  $A$. Then, from Corollary \ref{C23} there is a closed, open and convex subset $R$ of $X(A)$ such that (3) $G_1\,\cap\, G_2 = R \cup \bigcup\limits_{\scriptstyle i \in I} {f^A_i }^{-1}(R)$. By (2) and (3) we conclude  that $((R_1\setminus R_2)\setminus\bigcup\limits_{\scriptstyle i \in I} {f^A_i}^{-1}(R_1))\cap \bigcup\limits_{\scriptstyle i \in I} {f^A_i}^{-1}(R_2)\subseteq R \cup \bigcup\limits_{\scriptstyle i \in I} 
{f^A_i}^{-1}(R)\,)$ and  $((R_1\setminus R_2)\setminus \bigcup\limits_{\scriptstyle i \in I}{f^A_i}^{-1}(R_1)) \cap \bigcup\limits_{\scriptstyle i \in I} {f^A_i}^{-1}(R_2) \cap \bigcup\limits_{\scriptstyle i \in I}{f^A_i }^{-1}(R)\,)= \emptyset$, from which we get that $((R_1\setminus R_2)\setminus \bigcup\limits_{\scriptstyle i \in I} {f^A_i}^{-1}(R_1) ) \cap \bigcup\limits_{\scriptstyle i \in I}{f^A_i}^{-1}(R_2)\subseteq R$. Therefore, we have that (4) $((R_1\setminus R_2)\setminus \bigcup\limits_{\scriptstyle i \in I }{f^A_i}^{-1}(R_1)) \cap \bigcup\limits_{\scriptstyle i \in I} {f^A_i}^{-1}(R_2)\subseteq R \cap ((R_1\setminus R_2) \setminus \bigcup\limits_{\scriptstyle i \in I} {f^A_i}^{-1}(R_1))$. 
From the hypothesis and the fact that  $R \cap (R_1\setminus R_2)\setminus \bigcup\limits_{\scriptstyle i \in I} {f^A_i }^{-1}(R_1))$ is a closed  subset of $(R_1\setminus R_2)\setminus \bigcup\limits_{\scriptstyle i \in I} {f^A_i }^{-1}(R_1)$
we conclude that $R \cap((R_1\setminus R_2)\setminus \bigcup\limits_{\scriptstyle i \in I} {f^A_i }^{-1}(R_1)) = (R_1\setminus R_2)\setminus \bigcup\limits_{\scriptstyle i \in I} {f^A_i}^{-1}(R_1)$ and hence, (5) $(R_1\setminus R_2) \setminus \bigcup\limits_{\scriptstyle i \in I} {f^A_i }^{-1}(R_1)\subseteq  R$. Furthermore, from the hypothesis there is $x \in X(A)$ such that (6) $ x \in (R_1\setminus R_2)\setminus \bigcup\limits_{\scriptstyle i \in I} {f^A_i}^{-1}(R_1))$ and  (7) $x \not\in \bigcup\limits_{\scriptstyle i \in I} {f^A_i }^{-1}(R_2)$.
Then, (5) and (6) imply that $x \in R$ and so, by (3) it follows that (8) $x \in G_1 \cap G_2$. On the other hand, by (6) we infer  that  $x \not\in R_1\cap R_2$,  $x \not \in \bigcup\limits_{\scriptstyle i \in I} {f^A_i }^{-1}(R_1) \cap \bigcup\limits_{\scriptstyle i \in I} {f^A_i}^{-1}(R_2)$ and $x \not\in ((R_2\setminus R_1)\setminus  \bigcup\limits_{ \scriptstyle i \in I} {{f^A}_i}^{-1}(R_2))\cap \bigcup\limits_{\scriptstyle i \in I} {f^A_i}^{-1}(R_1)$. Besides, by  
 (7) we have that $x \not\in ((R_1\setminus R_2) \setminus \bigcup\limits_{\scriptstyle i \in I} {f^A_i}^{-1}(R_1)) \cap \bigcup\limits_{\scriptstyle i \in I}{f^A_i}^{-1}(R_2)$. These las assertions and (2) allow us to conclude that  $x \not\in  G_1 \cap G_2$, which contradicts (8). Therefore,  $\varphi_1 \cap \varphi_2$ is not a principal congruence on $A$.
\end{proof}

\

Bearing in mind the above results, our next task is to characterize the principal congruences on $n-$valued {\L}ukasiewicz
algebras without negation (or ${\rm LM}_{n}-$algebras) when we consider them as  ${\rm LM}_{\theta}-$algebras in the case that $\theta$ is an integer $n$, $n \geq 2$. It is well-known that each congruence on an ${\rm LM}_{n}-$algebra is a  
$\theta-$congruence on this algebra.  For this aim, first we wil determine the following properties of the ${\rm l}_{n}$P--spaces.

\

\begin{propo}\label{P232}
Let $( X,\{f_1,\ldots,f_{n-1}\})$ be an ${\rm l}_{n}$P--space. Then  condition {\rm (LP6)} is equivalent to any of these conditions:

\begin{itemize}
\item[{\rm(l$_{n}$P6)}] $X = \bigcup\limits_{\scriptstyle i =1}^{n-1}f_i(X)$, 
\item[{\rm(l$_{n}$P7)}] if $Y, Z$ are subsets of $X$ and $f_i^{-1}(Y)= f_i^{-1}(Z)$ for all $i$, $1\leqslant i\leqslant n-1$, then $Y=Z$, 
\item[{\rm (l$_{n}$P8)}] for each  $x \in X$, there is $i_0 $, $1 \leqslant i_0 \leqslant n-1$, such that $x = f_{i_0}(x)$.
\end{itemize}
\end{propo}

\begin{proof} 
(lP6)\, $\Leftrightarrow$\, (l$_{n}$P6): By (lP1), $X$ is a Hausdorff and compact space, from which it follows, by  (lP2), that  $f_i: X \fun X$ is a closed function for all  $i$, $1 \leq i \leq n-1$. Therefore,  $\bigcup\limits_{\scriptstyle i =1}^{n-1}f_i(X)$  is a closed subset of $X$ and so,  $\overline{\bigcup\limits_{\scriptstyle i =1}^{n-1}f_i(X)} = \bigcup\limits_{\scriptstyle i =1}^{n-1}f_i(X)$.

\vspace{2mm}

To prove the other equivalences is routine.
\end{proof}
 
\begin{cor}\label{PN31}{\rm(\cite[Proposition 3.1]{FPZ1})} 
Let $(X,\{f_1,\ldots, f_{n-1}\})$ be an ${\rm l}_{n}$P--space. Then $X$ is the cardinal sum of a family of chains, each of which has at most $n-1$ elements.
\end{cor}

\begin{proof} 
By  (lP11),  $X$ is the cardinal sum of the sets $[\{f_i(x): 1 \leq i \leq n-1\}) \cup  (\{f_i(x): 1 \leq i \leq n-1 \}]$,  $x \in X$. From (lP3), (lP4), (lP5) and  (l$_n$P8)  we infer that for each  $x \in X$, the set $[\{f_i(x): \, 1 \leq i \leq n-1\}) \cup  (\{f_i(x): 1 \leq i \leq n-1 \}] = \{f_i(x): 1 \leq i \leq n-1\}$ is a maximal chain in $X$ and so, the proof is complete. 
 \end{proof}

\

Now we are going to introduce the notion of modal subset of an ${\rm l}_{\theta}$P--space. These subsets play a fundamental role in the characterization of the con\-gruences on ${\rm Lk}_{n}-$algebras, as we will show next.

\begin{defi}\label{D321} Let $( X,\{f_i)\}_{i \in I})$ be an  ${\rm l}_{\theta}$P--space. A subset $Y$ of $X$ is modal if  $Y = {f_i}^{-1}(Y)$ for all $i\in I$.
\end{defi}
 
\vspace{2mm}

In order to reach our goal we will show the following lemmas.
   
\vspace{2mm}

\begin{lem}\label{PN34} 
Let $( X,\{f_1,\ldots,f_{n-1}\})$ be an ${\rm l}_{n}$P--space and let $Y$ be a non empty set of $X$. Then the following  conditions are equivalent: 

\begin{itemize}
\item[{\rm (i)}] $Y$ is semimodal,
\item[{\rm (ii)}] $Y$ is modal, 
\item[{\rm (iii)}] $Y$ is a cardinal sum of maximal chains in $X$.
\end{itemize}
\end{lem}

\begin{proof}
(i)\ $\Rightarrow$\ (ii):  Suppose that $f_{i}(z)\in Y$. By  (l$_n$P8) we have that $f_{i_0}(z) = z$ for some  $i_0$, $1\leq i_0 \leq  n-1$. Then, from the hypothesis and (lP5) we infer that $z=f_{i_0}(f_{i}(z))\in Y$.  Therefore, ${f_i}^{-1}(Y) \subseteq Y$ for all $i$, $1\leq i\leq n-1$. The other inclusion follows immediately. 

\vspace{2mm}

(ii)\ $\Rightarrow$\ (iii):  For each $y \in Y$, let $C_y= \{f_i(y): 1\leq i \leq n\}$. Then, taking into account the proof of Corollary  \ref{PN31} and (l$_n$P8) it follows that $C_y$ is a maximal chain in $X$ to which $y$ belongs. Therefore,  $Y \subseteq \bigcup\limits_{\scriptstyle y \in Y} C_y $. On the other hand, $\bigcup\limits_{\scriptstyle y \in Y} C_y = \bigcup\limits_{i =1}^{n-1} f_i(Y)$ and since $Y$ is modal, we conclude that $\bigcup\limits_{\scriptstyle y \in Y} C_y \subseteq Y$, which completes the proof.

\vspace{2mm}

(iii)\ $\Rightarrow$\ (i): It  follows from  (lP5) and the fact that a subset $C$ of $X$ is a chain if and only if $C = \{f_i(x): 1\leq i \leq  n-1\}$, for some $x \in X$.
\end{proof}

\begin{lem}\label{C312}  
Let $(X,\{f_1,\ldots,f_{n-1}\})$ be an ${\rm l}_{n}$P--space and let $Y$ be a modal subset of $X$. Then $X\setminus Y$ is  modal.
\end{lem}

\begin{proof} 
It is a direct consequence of  Definition \ref{D321}.
\end{proof}

\begin{theo}\label{T21b} 
Let $A$ be an LM$_n-$algebra and let ${\L}(A)$ be the l$_{n}$P--space associated with $A$. Then, it holds:

\begin{itemize}
\item[{\rm(ii)}] the lattice ${\cal C}_{M}(X(A))$ of all closed and modal subsets of $X(A)$ is isomorphic to the dual lattice $Con_{Lk_{n}}(A)$ of all {\rm Lk}$_{n}$-congruences on $A$ and the isomorphism is the function $\Theta_{M}: {\cal O}_{M}(X(A)) \fun Con_{Lk_{n}}(A)$, defined as in {\rm (A4)}.

\item[{\rm(ii)}] The lattice ${\cal O}_{M}(X(A))$ of all open and modal subsets of $X(A)$ is isomorphic to the lattice $Con_{Lk_{n}}(A)$ of all {\rm Lk}$_{n}$-congruences on $A$ and the isomorphism is the function $\Theta_{OM}: {\cal O}_{M}(X(A)) \fun Con_{Lk_{n}}(A)$, defined as in Theorem \ref{T21}.
\end{itemize}
\end{theo}

\begin{proof}
(i):\, It follows from (A4) and Lemma \ref{PN34}.

\vspace{2mm}

(ii):\, It is a consequence of Theorem \ref{T21} and  Lemmas \ref{PN34} and  \ref{C312}. 
\end{proof}

\

In the sequel, we take into account the well-known fact that Priestley duality provides an isomorphism between the lattices
${\cal F}(L)$ of all filters of a bounded distributive  lattice $L$ and that of ${\cal C}_I(X(L))$ of all closed and increasing subsets of $X(L)$. Under this isomorphism, any $F \in {\cal F}(L)$ corresponds to the increasing closed 
subset $Y_F = \bigcap \{\sigma_L(a): a \in F \}$, and any $Y \in {\cal C}_I(X(L))$ corresponds to the filter $F_Y = \{a \in L: Y \subseteq \sigma_L(a)\}$,  and   ${\bf\Theta}(F)=  \Theta(Y_F)$ and $\Theta(Y)= {\bf\Theta} (F_Y)$, where  $\Theta(Y)$ is  defined as in (A4) for all $Y \in {\cal C}_I(X(L))$ and ${\bf \Theta}(F)$ is the congruence associated with $F$.

\vspace{2mm}

\begin{propo}\label{P22b}
Let  $A$ be an ${\rm LM}_{n}-$algebra and let ${\L}_{n}(A)$ be the ${\rm l}_{n}$P--space associated with $A$. Then the following conditions are equivalent for all $a,b \in A$, $a \leqslant b$:  

\begin{itemize}
\item[\rm{(i)}] $\Theta_{OM}(G)= \Theta(a,b)$ and $G \in {\cal O}_{M}(X(A))$,
\item[\rm{(ii)}] $G$ is the least element of ${\cal O}_{M}(X(A))$,  ordered by inclusion,  which contains $\sigma_A(b) \setminus \sigma_A(a)$,
\item[\rm{(iii)}] $G =\bigcup\limits_{i =1}^{n-1} {f_i^A }^{-1}(\sigma_A(b)\setminus \sigma_A(a))$,
\item[\rm{(iv)}] $G =\sigma_A (\bigvee\limits_{j =1}^{n-1}(\phi_i b \wedge \overline{\phi}_i a))$,
\item[\rm{(v)}] $\Theta_{OM}(G) = {\bf \Theta}([\bigwedge\limits_{i =1}^{n-1}(\overline{\phi}_{i}b \vee \phi_{i}a )))$,
\end{itemize}
\end{propo}

\begin{proof}  
(i)\, $\Leftrightarrow$\, (ii): It follows from Theorem \ref{T21b} using the same argument as in Proposition \ref{C21}.  

\vspace{2mm}

(ii)\, $\Leftrightarrow$\, (iii): 
By (l$_n$lP8) we infer that $\sigma_A(b)\setminus \sigma_A(a)\subseteq \bigcup\limits_{i =1}^{n-1} {f_i^A }^{-1}( \sigma_A(b) \setminus \sigma_A(a))$. Furthermore, by (lP2) and (lP5) we have that $\bigcup\limits_{i =1}^{n-1} {f_i^A }^{-1}(\sigma_A(b)\setminus \sigma_A(a))\in {\cal O}_{M}(X(A))$. On the other hand, if $H \in{\cal O}_{M}(X(A))$ and $\sigma_A(b) \setminus \sigma_A(a) \subseteq H$, since $H$ is modal, we conclude that $\bigcup\limits_{i =1} ^{n-1} {f_i^A }^{-1}(\sigma_A(b) \setminus \sigma_A(a) \subseteq H$. Therefore,  
$G = \bigcup\limits_{i =1}^{n-1} {f_i^A }^{-1}( \sigma_A(b) \setminus \sigma_A(a)$ if and only if $G$ verifies (ii).

\vspace{2mm}

(iii)\, $\Leftrightarrow$\, (iv):  Taking into account that $\sigma_A$ is an ${\rm LM}_{n}-$isomorphism we have that    
$\bigcup\limits_{i =1}^{n-1} {f_i^A }^{-1}(\sigma_A(b) \setminus \sigma_A(a)) = \bigcup\limits_{i =1}^{n-1}({f_i^A }^{-1}(\sigma_A(b)) \cap (X(A) \setminus {f_i^A }^{-1}(\sigma_A(a))) = \bigcup\limits_{i =1}^{n-1} (\phi_i^{X(A)} \sigma_A(b)   \cap {\overline{\phi}}_i^{X(A)} \sigma_A(a))= \bigcup\limits_{i =1}^{n-1} (\sigma_A(\phi_i b) \cap \sigma_A(\overline{\phi}_i a)) = \bigcup\limits_{i =1}^{n-1} \sigma_A(\phi_i b \wedge \overline{\phi}_i a) = \sigma_A(\bigvee\limits_{j =1} ^{n-1} (\phi_i b \wedge \overline{\phi}_i a)$, from which the proof is complete. 

\vspace{2mm}

(iv)\, $\Rightarrow$\, (v): Bearing in mind that $\Theta_{OM} (G) = \Theta_{M}(X(A)\setminus G)$ and that $X (A)\setminus G = \sigma_A (\bigwedge\limits_{i =1}^{n-1}(\overline{\phi}_{i}b \vee \phi_{i}a ))$, by Theorem \ref{T21b} we infer that $\Theta_{OM}(G)= {\bf\Theta}([\bigwedge\limits_{i =1}^{n-1}(\overline{\phi}_{i}b \vee \phi_{i}a)))$.

\vspace{2mm}

(v)\, $\Rightarrow$\, (iv): Since $\sigma_A (\bigwedge\limits_{i =1}^{n-1} (\overline{\phi}_i b \vee \phi_i a))=\bigcap\{\sigma_A(x): x\in [\bigwedge\limits_{i =1}^{n-1} (\overline{\phi}_i b \vee \phi_i a))\}$,  by
 Theorem  \ref{T21b} we have that ${\bf\Theta}([\bigwedge\limits_{i =1}^{n-1}(\overline{\phi}_{i}b \vee \phi_{i}a)))= \Theta_{M}(\sigma_A(\bigwedge\limits_{i =1}^{n-1}(\overline{\phi}_{i}b \vee \phi_{i}a )))$. Besides, by  Theorem \ref{T21b} we infer that  
$\Theta_{M}(\sigma_A(\bigwedge\limits_{i =1}^{n-1}(\overline{\phi}_{i}b \vee \phi_{i}a ))) = \Theta_{O M}(\sigma_A ( \bigvee\limits_{j =1}^{n-1} (\phi_i b \wedge \overline{\phi}_i a)))$. Hence, $\Theta_{O M}(\sigma_A (\bigvee\limits_{j =1} ^{n-1} (\phi_i b \wedge \overline{\phi}_i a)))= {\bf\Theta}([\bigwedge\limits_{i =1}^{n-1}(\overline{\phi}_{i}b \vee \phi_{i}a)))$. From this last equality and the hypothesis we conclude that $\Theta_{OM}(G)= \Theta_{O M}(\sigma_A ( \bigvee\limits_{j =1}^{n-1} (\phi_i b  \wedge \overline{\phi}_i a)))$ and so, by Theorem \ref{T21b} we get $G = \sigma_A (\bigvee\limits_{j =1}^{n-1} (\phi_i b \wedge \overline{\phi}_i a))$.
\end{proof}

\begin{propo}\label{P312} 
Let $A$ be an LM$_n-$algebra and let ${\L}(A)$ be the l$_{n}$P--space associated with $A$. Then the following conditions are equivalent:

\begin{itemize}
\item[\rm{(i)}] $\vartheta$  is a principal ${\rm Lk}_{n}-$congruence,  
\item[ \rm{(ii)}] $\vartheta =\Theta_{OM}(G)$, where $G =\bigcup\limits_{i =1}^{n-1}{f_i^A }^{-1}(R)$ and  $R$ is a closed, open and convex subset of $X(A)$,   
\item[ \rm{(iii)}] $\vartheta =\Theta_{OM}(G)$, where $G$ is a closed, open and modal subset of  $X(A)$.
\end{itemize}
\end{propo}

\begin{proof}

(i)\, $\Leftrightarrow$\,(ii): It is an immediate consequence of  Proposition \ref{P22b} and (A7).

\vspace {2mm}

(ii)\, $\Rightarrow$\,(iii): From the hypothesis, (lP2) and  (lP5) we have that $G$ is closed, open and modal subset of $X(A)$.

\vspace {2mm}

(iii)\, $\Rightarrow$\,(ii):  From Proposition \ref{PN34}, we have that $G$  is the cardinal sum of maximal chains and so, $G$ is a convex subset of $X(A)$. Hence, from the hypothesis by taking $R=G$ we conclude the proof. 
\end{proof}

\begin{cor}\label{C314}
Let $A$ be an Lk$_n-$algebra. Then, the intersection of two principal ${\rm Lk}_{n}-$congruence on $A$ is a  principal one. 
\end{cor}

\begin{proof} 
It is a direct consequence of Proposition \ref{P312} and Theorem \ref{T21b}.
\end{proof}

\section{\textbf \large Boolean congruences}\label{s04}\indent 

Next, our attention is focus on determine the Boolean congruences and ${\theta}-$con\-gru\-ences on ${\rm LM}_{\theta}-$algebras bearing in mind the topological duality for them established in Section \ref{s01}. In order to do this, we will start studying certain subsets of ${\rm l}_{\theta}$P--spaces which will be fundamental   
to reach our goal.

\begin{propo}\label{L321} 
Let  $A$ be  an ${\rm LM}_{\theta}-$algebra and let  ${\L}_{\theta}(A)$ be the  ${\rm l}_{\theta}$P--space  associated with $A$. Then for each $Y \subseteq X(A)$ holds:

\begin{itemize}
\item[{\rm(i)}] $\Theta_{OS}(Y)$ is a Boolean congruence on $A$ if and only if $Y$ is a closed and open subset of $X(A)$ such that $Y$ and $X (A)\setminus Y$ are semimodal, where  $\Theta_{OS}(Y)$ is defined  as in Theorem \rm{\ref{T21}}.

\item[{\rm(ii)}] $\Theta_{O\theta}(Y)$ is a Boolean $\theta-$congruence on $A$ if and only if $Y$ is a closed and open subset of $X(A)$ such that $Y$ and $X (A)\setminus Y$ are $\theta-$subsets of $X(A)$, where $\Theta_{O\theta}(Y)$ is defined as in Theorem \rm{\ref{T21}}.
\end{itemize}
\end{propo}

\begin{proof}
It is a direct consequence of Theorem \ref{T21}.
\end{proof}

\

Now, we will recall two concepts which will be used in Proposition \ref{P321}. 
Let $Y$ be a topological space and $y_0 \in Y$. A net in a space $Y$ is a map $\varphi : D \to Y$ of some directed set $(D, \prec)$. Besides, we say that $\varphi$ converges to $y_0$ (written
$\varphi \rightarrow y_0$) if for all neighborhood $U(y_0)$ of $y_0$ there is $a\in D$ such that 
for all $a\prec b$, $\varphi(b)\in U(y_0)$. 

\

\begin{propo}\label{P321} 
Let $(X,\{f_i\}_{i \in I})$ be an  ${\rm l}_{\theta}$P--space and let $Y$ be a closed, open and semimodal subset of $X$. Then $X\setminus Y$ is semimodal. 
\end{propo}

\begin{proof} 
Suppose that there is  (1) $x \in X \setminus Y$  such that $f_{i_0}(x)\in Y$ for some $i_0\in I$. Since (2)  $Y$ is semimodal,  we infer by (lP5) that (3) ${f}_{i}(x)\in Y$ for all $i \in I$. Taking into account that $Y$ is a closed subset of $X$ it follows by  (1) that $X \setminus Y$ is a neighborhood of $x$. Then, by (2) we have that $(X \setminus Y) \cap \,\bigcup\limits_{\scriptstyle i\in I} f_i(Y)= \emptyset$ and so, $x \not\in \overline{\bigcup\limits_{\scriptstyle i\in I} f_i(Y)} $. From this last assertion and (lP6) we conclude that $x \in \overline{\bigcup\limits_{\scriptstyle i\in I} f_i(X\setminus Y)} $, from which it follows that there exists  a net $\{ x_d\}_{d\in D} \subseteq X \setminus  Y $ and (4) ${f}_{ i_d}(x_d) \rightarrow x$. Therefore, there exists $d_0 \in D$ such that $\{ f_{i_d} (x_d): d_0 \prec d, \, {d\in D}\} \subseteq X \setminus  Y $. From (2) and (lP5),   $\{ f_{i} (x_d): d_0 \prec d, \, {d\in D}, \, i \in I\} \subseteq X \setminus  Y $. On the other hand, by (4), (lP2) and (lP5) we have that ${f}_{ i}(x_d) \rightarrow f_i(x)$ for all $i\in I$  and from the fact that $X \setminus Y$ is closed, we conclude that  ${f}_{i}(x)\in X \setminus Y$ for all $i \in I$, which contradicts (3).
\end{proof}

\begin{propo}\label{L322N}
Let $(X,\{f_i\}_{i \in I})$ be an ${\rm l}_{\theta}$P--space. Then for each $x \in X$, the set $[\{f_i(x)\}_{i \in I}) \cup  (\{f_i(x)\}_{i \in I}]$ is convex. If $I$ has least element $0$ and greatest element $1$, then for each $x \in X$ the set $[f_0(x), f_1(x)]$ is convex.
\end{propo}

\begin{proof} 
Let $y,z,w \in X$ be such that $y \leq z \leq w$ and let $y,w \in [\{f_i(x)\}_{i \in I}) \cup  (\{f_i(x)\}_{i \in I}]$. Then, by (lP3) and (lP5) we have that $f_i(y)= f_i(z) = f_i(w)= f_i(x)$ for all $i \in I$. This statement and (lP8) imply that $z \in [\{f_i(x)\}_{i \in I}) \cup  (\{f_i(x)\}_{i \in I}]$ and so, $[\{f_i(x)\}_{i \in I}) \cup  (\{f_i(x)\}_{i \in I}]$ is a convex set. 

On the other hand, if $I$ has least and greatest element, for each  $x \in X$ we have that $[f_0(x), f_1(x)] =[\{f_i(x)\}_{i \in I}) \cup  (\{f_i(x)\}_{i \in I}]$ and the proof is concluded.
\end{proof}

\begin{propo}\label{L322}
Let $(X,\{f_i\}_{i \in I})$ be an ${\rm l}_{\theta}$P--space and let $Y \subseteq X$. Then the following conditions are  equivalent: 
\begin{itemize}
\item[{\rm(i)}] $Y$ is modal,
\item[{\rm(ii)}] $Y =\bigcup\limits_{\scriptstyle  y\in Y} ([\{f_i(y)\}_{i \in I})\cup (\{f_i(y)\}_{i \in I}])$. Besides, if $I$ has least and greatest element, $Y =\bigcup\limits_{\scriptstyle  y\in Y} [{f_0}(y),{ f_1}(y)]$.  
\end{itemize}
\end{propo}

\begin{proof}
(i)\, $\Rightarrow$\, (ii): By (lP8),  $Y \subseteq \bigcup\limits_{\scriptstyle y\in Y} ([\{f_i(y)\}_{i \in I}) \cup  (\{f_i(y)\}_{i \in I}])$. On the other hand, if $ z \in [\{f_i(y)\}_{i \in I}) \cup  (\{f_i(y)\}_{i \in I}]$, for some $y \in Y$, then by (lP3) and (lP5) we have that $f_i(z) = f_i(y)$ for all $i\in I$ and so, from the hypothesis we conclude that $z \in Y$.
 
\vspace{2mm} 

(ii)\, $\Rightarrow$\, (i): From (lP3), (lP5) and (lP8),  ${f_i}^{-1}(([\{f_i(y)\}_{i \in I}) \cup  (\{f_i(y)\}_{i \in I}]))= [\{f_i(y)\}_{i \in I}) \cup (\{f_i(y)\}_{i \in I}]$ for all $i\in I$ and so, $Y = {f_i}^{-1}(Y)$ for all $i\in I$.
In case that $I$ has least and greatest element, it holds that $\bigcup\limits_{\scriptstyle  y\in Y} ([\{f_i(y)\}_{i \in I})\cup (\{f_i(y)\}_{i \in I}])=\bigcup\limits_{\scriptstyle  y\in Y} [{f_0}(y),{ f_1}(y)]$ and hence, $Y =\bigcup\limits_{\scriptstyle  y\in Y}[{f_0}(y),{ f_1}(y)]$.
\end{proof}

\begin{cor}\label{C321}
Let $(X,\{f_i\}_{i \in I})$ be an ${\rm l}_{\theta}$P--space. Then the following conditions hold, for each modal subset $Y$ of  $X$:
\begin{itemize}
\item[{\rm(i)}]  $X \setminus Y$ is modal,
\item[{\rm(ii)}] $Y$ is convex,
\item[{\rm(iii)}]  $Y$ is increasing and decreasing. 
\end{itemize}
\end{cor}

\begin{proof}
(i) and (iii): They are an immediate consequence of Proposition \ref{L322} and (lP11).

\vspace{2mm}

(ii): It follows from Propositions \ref{L322}, \ref{L322N} and (lP11).  
\end{proof}

\begin{propo}\label{P322}
Let $(X,\{f_i\}_{i \in I})$ be an ${\rm l}_{\theta}$P--space and let $Y$ be a closed and open subset of $X$. Then the  following conditions are equivalent:

\begin{itemize}
\item[{\rm(i)}] $Y$ is a ${\theta}-$subset,
\item[{\rm(ii)}] $Y$ is semimodal,
\item[{\rm(iii)}] $Y$ is modal.
\end{itemize}
\end{propo}

\begin{proof}
(i)\, $\Rightarrow$\, (ii): It follows immediately.

\vspace{2mm}

(ii)\, $\Rightarrow$\, (iii): From (lP8) we have that $Y\subseteq \bigcup\limits_{\scriptstyle  y\in Y} ([\{f_i(y)\}_{i \in I}) \cup (\{f_i(y)\}_{i \in I}])$.  On the other hand, if $z \in [\{f_i(y)\}_{i \in I})\cup (\{f_i(y)\}_{i \in I}]$  for some $y \in Y$, by (lP3) and (lP5) we conclude  that $f_i(z) = f_i(y)$ for all $i \in I$. Furthermore, from the hypothesis we infer that $f_i(z)\in Y$ for all $i \in I$ and then, Proposition \ref{P321} allows us to assert that $z \in Y$. Therefore, $Y =\bigcup\limits_{\scriptstyle  y\in Y}([\{f_i(y)\}_{i \in I})\cup (\{f_i(y)\}_{i \in I}])$ and by 
Proposition \ref{L322} we have that $Y$ is modal.

\vspace{2mm}

(iii)\, $\Rightarrow$\, (i): It follows immediately that $Y$ is semimodal. Furthermore, since $Y$ is a closed and open subset of $X$, from Proposition \ref{P321} we get that $X \setminus Y$ is  semimodal, which implies that  (1) $\overline{\bigcup\limits_{\scriptstyle i\in I} f_i( Y)} \subseteq Y$ and (2) $\overline{\bigcup\limits_{\scriptstyle i\in I} f_i(X\setminus Y)} \subseteq X \setminus Y$.  On the other hand, by (lP6) it holds (3) $X = \,\overline{\bigcup\limits_{ \scriptstyle i\in I} f_i( Y)}\, \cup \, \overline{\bigcup\limits_{\scriptstyle i\in I} f_i(X\setminus Y)}$. From (2) and (3) we infer that $Y \subseteq \overline{\bigcup\limits_{\scriptstyle i\in I} f_i( Y)}$ and so,  by (1) we conclude that
 $Y =\,\overline{\bigcup\limits_{\scriptstyle i\in I} f_i( Y)}$.
\end{proof}

\begin{cor}\label{C322}
Let $(X,\{f_i\}_{i \in I})$ be a ${\rm l}_{\theta}$P--space. Then the following conditions hold for each closed, open  
and $\theta$-subset $Y$ of $X$:
\begin{itemize}
\item[\rm{(i)}] $X \setminus Y$ is a $\theta$-subset,  
\item[\rm{(ii)}] $Y$ and $X \setminus Y$ are convex subsets of $X$.
\end{itemize}
\end{cor}

\begin{proof}   
\rm{(i)}: It follows from Proposition \ref{P322} and Corollary \ref{C321}. 

\vspace{2mm} 

\rm{(ii)}: It is a direct consequence of Proposition \ref{P322} and Corollaries \ref{C321} and \ref{C322}.
\end{proof}

\begin{propo}\label{P323}
Let  $A$ be an  ${\rm LM}_{\theta}-$algebra and let ${\L}_{\theta}(A)$ be the ${\rm l}_{\theta}$P--space associated with $A$.  Then for each $Y\subseteq X(A)$ the following conditions are equivalent: 

\begin{itemize}
\item[ \rm{(i)}]  $Y$ is a closed, open and modal subset of $X(A)$,
\item[\rm{(ii)}]  there is  $b\in C(A)$ such that  $Y =\sigma_A( b)$.
\end{itemize}
\end{propo}

\begin{proof}
(i)\, $\Rightarrow$\, (ii): From the hypothesis and item (iii) in Corollary \ref{C321}, $Y \in D(X(A))$. Hence, there exists $a \in A$ such that (1) $Y =\,\sigma_A(a)$. Taking into account that $Y = {f^A_i}^{-1}(Y)$  for all $i \in I$,  $\sigma_A$ is  an ${\rm LM}_{\theta}-$isomorphism and (A1) we conclude that $Y = \sigma_A(\phi_i a)$ for all $i \in I$  and so, by (1) we have that $a = \phi_i(a)$ for all $i \in I$. This statement and (L6) imply that  $a \in C(A)$. 

\vspace{2mm}

(ii)\, $\Rightarrow$\, (i): From the hypothesis it follows that $Y \in D(X(A))$. Besides, for all  $i \in I$, $\sigma_A(\phi_i b) = \,\phi_i^{D(X(A))}(\sigma_A(b))$ which by (A1) entails $\sigma_A(\phi_i b) = \, {f^A_i}^{-1}(Y) $. Since $b \in C(A)$ by (L6) we have that  $\phi_i b =b$ for all $i \in I$. Therefore, ${f^A_i}^{-1}(Y) = Y$  for all $i \in I$ which completes the proof.
\end{proof}

\

Let $(X,\{f_i\}_{i \in I})$ be a ${\rm l}_{\theta}$P--space. We will denote by ${\cal C \cal O}_{M}(X)$  the Boolean lattice of all closed, open and modal subsets of $X$.

\vspace{2mm}

\begin{cor}\label{C323} 
Let  $A$ be an  ${\rm LM}_{\theta}-$algebra and let ${\L}_{\theta}(A)$ be the ${\rm l}_{\theta}$P--space associated with $A$. Then ${\cal C \cal O}_{M}(X(A))$ is isomorphic to the Boolean lattice $C(A)$.
\end{cor}

\begin{proof} 
Proposici\'on \ref{P323} allows us to assert that the restriction of $\sigma_A$ to $C(A)$ is a Boolean isomorphism.
\end{proof}

\

The above results allow us to obtain the description of  Boolean congruences we were looking for.

\vspace{2mm}

\begin{theo}\label{T321}
Let  $A$ be an  ${\rm LM}_{\theta}-$algebra and let ${\L}_{\theta}(A)$ be the ${\rm l}_{\theta}$P--space associated with $A$. Then the lattice  ${\cal C \cal O}_{M}(X(A))$ is isomorphic to the lattice {\rm(}dual lattice{\rm)} $Con_{bLM_{\theta}}(A)$ of Boolean congruences on $A$, and the isomorphism  $\Theta_{OM}$ {\rm(}$\Theta_{CM}${\rm)} is the restriction of  $\Theta_{OS}$ {\rm(}$\Theta_{S}${\rm)} to ${\cal C \cal O}_{M}(X(A))$, where these functions are defined as in Theorem {\rm \ref{T21}} {\rm ( in (A4))} respectively.  
\end{theo}

\begin{proof}  
If $Y$ is a closed, open and modal subset of $X(A)$, then from Propositions \ref{L321}, \ref{P321} and  \ref{P322} 
we have that $\Theta_{OM}(Y)$ is a Boolean congruence on $A$. Conversely, let $\varphi \in Con_{bLk_{\theta}}(A)$. Then, by Theorem \ref{T21} and Proposition \ref{L321} we infer that there is  a closed and open subset $Y$ of $X(A)$ such that $Y$,  $X(A)\setminus Y$ are  semimodal and  $\varphi = \Theta_{OS}(Y)$.
These assertions and Proposition \ref{P322} imply that $ Y\in {\cal C \cal O}_{M}(X(A))$ and $\Theta_{OS}(Y) = \Theta_{OM}(Y)$ and so, by Theorem \ref{T21} we conclude the proof. 

On the other hand, taking into account that $Y \in {\cal C \cal O}_{M}(X(A))$ if and only if  $X(A)\setminus Y \in {\cal C \cal O}_{M}(X(A))$  and that $\Theta_{CM}(Y) = \Theta_{OM}(X(A) \setminus Y)$ we infer that $\Theta_{CM}$  establishes an isomorphism between  ${\cal C \cal O}_{M}(X(A))$  and the dual of $Con_{bLM_{\theta}}(A)$. 
\end{proof}

\begin{cor}\label{C324} 
Let  $A$ be an  ${\rm LM}_{\theta}-$algebra and let ${\L}_{\theta}(A)$ be the ${\rm l}_{\theta}$P--space associated with $A$. If $\varphi$ is a congruence on $A$ then the following conditions are equivalent:

\begin{itemize}
\item[ \rm{(i)}] $\varphi$ is a Boolean congruence,
\item[\rm{(ii)}] $\varphi$ is a Boolean $\theta-$congruence.
\end{itemize}
\end{cor}

\begin{proof}
(i)\, $\Rightarrow$\, (ii): It is a direct consequence of Theorem \ref{T321} and Propositions \ref{P322} and \ref{L321}.

\vspace{2mm}

(ii)\, $\Rightarrow$\, (i): It follows immediately.
\end{proof}

\begin{cor}\label{C325} 
Let  $A$ be an  ${\rm LM}_{\theta}-$algebra and let ${\L}_{\theta}(A)$ be the ${\rm l}_{\theta}$P--space associated with $A$. Then each Boolean congruence on  $A$ is both a principal congruence 
 and $\theta-$congruence on  $A$. 
\end{cor}

\begin{proof} 
Let $\varphi$ be  a Boolean ${\rm LM}_{\theta}-$congruence on $A$. Then, Theorem \ref{T321} implies that there is $G \in {\cal C \cal O}_{M}(X(A))$ such that $\varphi =  \Theta_{OM}(G)$. Furthermore, Proposition \ref{P322} and  Corollary \ref{C321} imply that  ${\cal C \cal O}_{M}(X(A)) \subseteq {\cal O}_{CS}(X(A))$. Hence, we have that $G \in {\cal O}_{CS}(X(A))$ and so,  $\Theta_{OM}(G) = \Theta_{OS}(G)=\varphi$. Since $G$ is modal, by Corollary \ref{C321} we infer that $G$ is convex. From this last assertion and Corollary \ref{C23} we conlude that $\varphi$ is a principal ${\rm LM}_{\theta}-$ congruence  on $A$. 

On the other hand, from Proposition \ref{P322} and  Corollary \ref{C321} we have that ${\cal C \cal O}_{M}(X(A)) \subseteq {\cal O}_{C\theta}(X(A))$ from which we get that $\Theta_{OM}(G) = \Theta_{O\theta}(G)=\varphi$. Hence, by  Proposition \ref{P322} we conclude that $G$  is a closed and  open  $\theta-$subset of $X(A)$ and so, from  Corollary \ref{C322} we infer that $X(A)\setminus G$ is a closed $\theta-$subset of $X(A)$. This statement means that $X(A)\setminus G =\overline 
{\bigcup\limits_{\scriptstyle i \in I} f^A_i (X(A)\setminus G)}$. On the other hand since  $X(A)\setminus G$ is modal, 
$X(A)\setminus G =\,\bigcap\limits_{\scriptstyle i \in I} {f^A _i}^{-1}(X(A)\setminus G)$. These last assertions imply that  $G = X(A)\setminus \overline {\bigcup\limits_{\scriptstyle i \in I} f^A_i(\bigcap\limits_{\scriptstyle i \in I} {f^A _i}^{-1}(X(A) \setminus G))}$ and since $G$ is convex, Corollary \ref{P26} allows us to conclude that $\varphi$ is a principal $\theta-$congruence on $A$.
\end{proof}

\begin{cor}\label{C326}
Let  $A$ be an ${\rm LM}_{\theta}-$algebra. Then the Boolean algebras \\ $Con_{bLM_{\theta}}(A)$ and $C(A)$ are isomorphic and therefore, $|Con_{bLM_{\theta}}(A)|= |C(A)|$,  where $|Z|$ denotes the cardinality of the set $Z$.
\end{cor}
 \begin{proof}  It is a direct consequence of Corollary \ref{C323} and  Theorem \ref{T321}. 
\end{proof}

\begin{cor}\label{C415} 
Let  $A$ be an ${\rm LM}_{\theta}-$algebra let ${\L}_{\theta}(A)$ be the ${\rm l}_{\theta}$P--space as\-so\-cia\-ted with $A$.
Then, Boolean congruences on $A$ are permutable. 
\end{cor}

\begin{proof}
Let $\varphi_1$, $\varphi_2 \in Con_{bLM_{\theta}}(A)$. Then, by Theorem \ref{T321} there are closed, open and  modal  subsets  $Y_1$, $Y_2$ of $X(A)$ such that  $\theta_S (Y_1) = \varphi_1$ and $\theta_S (Y_2) = \varphi_2$.
Suppose now that  $(x, y) \in \varphi_2 \circ \varphi_1$. Hence, there is  $z \in A$ such that  
$(x,z)\in \varphi_1$ and  $(z,y) \in \varphi_2$ and so, from Theorem \ref{T321} we have that 
$\sigma_A (x)\cap Y_1 = \sigma_A (z)\cap Y_1$  and  $\sigma_A (y) \cap Y_2 = \sigma_A (z) \cap Y_2 $. 
These statements imply that $\sigma_A (x) \cap (Y_1 \cap Y_2) = \sigma_A (y)\cap (Y_1 \cap Y_2)$. On the other hand, since $Y_1, Y_2 \in {\cal C \cal O}_{M}(X(A))$, by  Corollary \ref{C321} and  Proposition  \ref{P323} we infer that $(\sigma_A (x) \cap (Y_1 \cap Y_2)) \cup (\sigma_A (x) \cap (Y_2 \setminus Y_1)) \cup (\sigma_A (y) \cap (Y_1 \setminus Y_2)) \in D(X(A))$ and so, $w = \sigma_A^{-1}((\sigma_A (x) \cap (Y_1 \cap Y_2)) \cup (\sigma_A (x) \cap (Y_2 \setminus Y_1)) \cup (\sigma_A (y) \cap (Y_1 \setminus Y_2))) \in A$. Furthermore, we have that  
 $\sigma_A (x) \cap Y_2=\sigma_{A}(w) \cap Y_2$ and $\sigma_{A}(w) \cap Y_1 = \sigma_A (y) \cap Y_1$, hence $(x,w) \in \varphi_2$ and  $(w,y) \in \varphi_1$. Therefore, $(x,y) \in \varphi_1 \circ \varphi_2$  from which  we conclude that  $\varphi_2 \circ \varphi_1\subseteq \varphi_1 \circ \varphi_2$. The other inclusion follows similarly. 
\end{proof}

\

Next, we will give another characterization of the Boolean congruences which will be useful in order to  determine some properties of them. 

\

\begin{lem}\label{L414} 
Let  $A$ be an  ${\rm LM}_{\theta}-$algebra and let ${\L}_{\theta}(A)$ be the ${\rm l}_{\theta}$P--space associated with $A$. Then ${\bf\Theta}([\phi_i a))$ is an congruence on $A$ for all $a \in A$ and for all  $i \in I$.
\end{lem}

\begin{proof} 
Since $\sigma_A(\phi_i(a))= {f^A_i}^{-1}(\sigma_A(a))$, by (lP5) we have that $\sigma_A(\phi_i(a))$ is a modal subset of $X(A)$ and so,  it is semimodal. Then, by (A4) we infer that $\Theta_S(\sigma_A(\phi_i(a))$ is a congruence  on  $A$. Bearing in mind the definition of $\Theta_S(\sigma_A(\phi_i(a)))$, we conclude that $\Theta_S(\sigma_A(\phi_i(a))= {\bf\Theta} ([\phi_i(a))$ which completes the proof. 
\end{proof}

\begin{propo}\label{P324} 
Let $A$ be an  ${\rm LM}_{\theta}-$algebra and $\varphi \in Con_{Lk_{\theta}}(A)$. Then the following conditions are equivalent:

\begin{itemize}
\item[ \rm{(i)}]  $\varphi$ is a Boolean congruence on $A$,
\item[\rm{(ii)}] there is $a \in A$ and $i\in I$ such that $\varphi= {\bf \Theta}([\phi_i a))$.
\end{itemize}
\end{propo}

\begin{proof}
(i)\, $\Rightarrow$\, (ii): From the hypothesis and Theorem \ref{T321} there is $Y\in {\cal C\cal O}_{M}(X(A))$ such that  
$\varphi = \Theta_{CM}(Y)$. Besides, from  Proposition \ref{P323} and (L6) there is  $a \in A$ and $Y =\sigma_A(\phi_i(a))$ for some $i\in I$. Therefore,  for all $b,c \in A$ we have that $(b,c)\in \Theta_{CM}(Y)$ if and only if $\sigma_A(b)\cap \sigma_A(\phi_i(a))=\sigma_A(c)\cap \sigma_A(\phi_i(a))$. Hence, taking into account that $\sigma_A$ is an isomorphism it follows that $(b,c)\in \Theta_{CM}(Y)$ if and only if $b\wedge \phi_i(a)=c \wedge \phi_i(a)$. From this last assertion we conclude that $\varphi={\bf\Theta}([\phi_i(a)))$. 

\vspace{2mm}

(ii)\, $\Rightarrow$\, (i): From Lemma \ref{L414} we infer that ${\bf\Theta}([\phi_i a))$ is a congruence on $A$ and ${\bf\Theta}([\phi_i a)) = \Theta_S(\sigma_A(\phi_i(a))$. Besides, from  Proposition \ref{P323},  $\sigma_A(\phi_i(a)) \\ \in {\cal C \cal O}_{M}(X(A))$ and so, by Theorem \ref{T321}  we have that  $\Theta_{CM}(\sigma_A(\phi_i(a))) = \Theta_S(\sigma_A(\phi_i(a)))$ is  a Boolean congruence on  $A$.  
\end{proof}

\begin{cor}\label{C327} 
Let $A$ be an  ${\rm LM}_{\theta}-$algebra. Then the Boolean congruences on $A$ are normal and regular.
\end{cor}

\begin{proof}
Let $\varphi$ be a Boolean congruence on $A$. Then, by Proposition \ref{P324} there is $a \in A$ and  $i \in I$ such that $\varphi = {\bf\Theta}([\phi_i a))$. Furthermore, for each $b \in A$ we have that $\overline b_\varphi=
\{(b\,\wedge \phi_ia)\vee c :\, c\in (\overline{\phi}_i a]\}$ where $\overline b_\varphi$ stands for the equivalence class of $b$ modulo $\varphi$. From this last assertion we infer that $\overline 0_\varphi= \left(\overline{\phi}_i a\right]$ and therefore, $\overline b_\varphi=\{(b\,\wedge \phi_ia)\vee c:\, c\in \overline 0_\varphi\}$ and   
$\left|\overline b_\varphi\right|= \left|\overline 0_\varphi\right|$ for all $b \in A$, which allows us to conclude the proof. 
\end{proof}

\

In what follows we will determine necessary and sufficient conditions for a principal congruence on an  ${\rm LM}_{\theta}-$algebra be a Boolean one. These are also sufficient conditions for the fact that the intersection of two principal ${\rm LM}_{\theta}-$congruences be  a principal one.

\begin{propo}\label{PAP324}
Let  $A$ be an  ${\rm LM}_{\theta}-$algebra and let ${\L}_{\theta}(A)$ be the ${\rm l}_{\theta}$P--space associated with $A$. Then the following conditions are equivalent for all $a,b\in A$ such that $a \leqslant b$:

\begin{itemize}
\item[\rm{(i)}]  $\Theta (a,b)$ is a Boolean congruence on $A$,
\item[\rm{(ii)}] $\bigcup\limits_{\scriptstyle i \in I}{f^A_i}^{-1}(\sigma_A(b)\setminus \sigma_A(a))$ is a closed subset of $X(A)$,  
\item[{\rm(iii)}] $\bigcup\limits_{\scriptstyle i \in I}\sigma_A(\phi_i b\wedge \overline{\phi}_i a)$ a closed subset of $X(A)$.
\end{itemize}
\end{propo}

\begin{proof}
(i)\, $\Rightarrow$\, (ii): From the hypothesis and Proposition \ref{C21} we have that $\Theta(a,b)=\Theta_{OS}(G)$, where $G =(\sigma_A(b)\setminus \sigma_A(a))\cup \bigcup\limits_{\scriptstyle i \in I} {f^A_i }^{-1}(\sigma_A(b)\setminus \sigma_A(a))$ and taking into account  Theorem \ref{T321} we infer that $G$ is a closed, open and modal subset of $X(A)$. This last assertion allows us to conclude  that $G =\bigcup\limits_{\scriptstyle i \in I} {f^A_i }^{-1}(\sigma_A(b) \setminus \sigma_A(a))$ and so, the  proof is completed.

\vspace{2mm}

(ii)\, $\Rightarrow$\, (i): Let (1) $G =\bigcup\limits_{\scriptstyle i \in I} {f^A_i}^{-1}(\sigma_A(b) \setminus \sigma_A(a))$. Then from the hypothesis, (lP2) and (lP5) we have that (2) $G \in {\cal C \cal O}_{M}(X(A))$, and as a consequence of  Theorem \ref{T321} we infer that $\Theta_{OM}(G)$ is a Boolean congruence on $A$. On the other hand,  from (2), Proposition \ref{P322} and Corollary \ref{C321} it follows that $G \in {\cal O}_{CS}(X(A))$ and so, 
(3) $\Theta_{OM}(G)=\Theta_{OS}(G)$. Besides, $\sigma_A(b) \setminus \sigma_A(a)\subseteq G$. Indeed, suppose that 
$(\sigma_A(b) \setminus \sigma_A(a)) \setminus G \not= \emptyset$. Since $G$ is a closed subset of $X(A)$, by (lP6) we have that $((\sigma_A(b)\setminus \sigma_A(a))\setminus  G)\cap \bigcup\limits_{\scriptstyle i \in I} {f^A_i }(X(A)) \not= \emptyset$, from which we infer that there are (4) $x \in (\sigma_A(b)\setminus \sigma_A(a))\setminus G$ and  $y \in X(A)$ such that $x = f_{i_0}(y)$ for some  $i_0 \in I$. Then by (lP5), $x = f_{i_0}(x)$ for some  $i_0 \in I$. This  statement, (1) and (4) imply that $x\in G$, which contradicts (4). On the other hand, it follows immediately that  $G$  is the least element of ${\cal O}_{CS}(X(A))$, ordered by inclusion such that $\sigma_A(b) \setminus \sigma_A(a)\subseteq G$. Hence, by Proposition \ref{C21} and (3) we conclude that $\Theta(a,b)=\Theta_{OM}(G)$. Therefore, $\Theta(a,b)$ is a Boolean congruence on $A$.
 
\vspace{2mm}

(ii)\, $\Leftrightarrow$\, (iii): It is a direct consequence of the fact that $\sigma_A$ is an iso\-mor\-phism.
\end{proof}

\begin{lem}\label{L323}
Let $(X,\{f_{i}\}_{i\in I})$ be an ${\rm l}_{\theta}$P--space. If $R \subseteq X$ is such that $R \subseteq \bigcup\limits_{\scriptstyle i \in I}{f_i}^{-1}(R)$, then $\bigcup\limits_{\scriptstyle i \in I}{f_i}^{-1}(R)=(R] \cup [R)$
\end{lem}

\begin{proof} 
Suppose that $f_i(x)\in R$ for some $i\in I$. Hence, by (lP8) we have that $x \in (R]$ or $x \in [R)$. Conversely, let $x \in (R] \cup [R)$. Then, there is  $y \in R$ such that $x \leqslant y$ or $ y \leqslant x$ and so, by (lP3) we infer that  $f_i(x) = f_i(y)$ for all $i \in I$. This assertion and the fact that from the hypothesis  
$f_{i_0}(y)\in R$ for some $i_0 \in I$ allow us to conclude that $x \in \bigcup\limits_{\scriptstyle i \in I} {f_i}^{-1}( R)$.
\end{proof}

\begin{lem}\label{L324} 
Let $(X,\{f_{i}\}_{i \in I})$ be an  ${\rm l}_{\theta}$P--space and  $R$ be a closed and open subset of $X$. Then the following conditions are equivalent:

\begin{itemize}
\item[\rm{(i)}] $R \subseteq \bigcup\limits_{\scriptstyle i \in I} {f_i} ^{-1}(R)$,
\item[\rm{(ii)}] $\bigcup\limits_{\scriptstyle i \in I}{f_i}^{-1}(R)$ is closed.
\end{itemize}
\end{lem}

\begin{proof}
(i)\, $\Rightarrow$\, (ii): From the hypothesis and Lemma \ref{L323} we have that $\bigcup\limits_{\scriptstyle i \in I}{f_i}^{-1}(R)= (R]\cup [R)$. Since $R$ is a closed subset of $X$ by (A6) we infer that $(R]\cup [R)$ is closed, which completes the proof.

\vspace{2mm}

(ii)\, $\Rightarrow$\, (i): From the hypothesis and the fact that $R$ is an open subset of $X$ we have that $R\setminus \bigcup\limits_{\scriptstyle i \in I} {f_i}^{-1}(R)$ is open. Suppose now that $R \setminus \bigcup\limits_{\scriptstyle i \in I}{f_i}^{-1}(R)\not=\emptyset$. Then, by (lP6) there are  $x\in X$ and  $i_{0}\in I$ such that  $f_{i_0}(x)\in R \setminus \bigcup\limits_{\scriptstyle i \in I}{f_i}^{-1}(R)$ and from (lP5) we conclude that $f_{i_0}(x)\in R $ and  
$f_{i}(x)\not\in R$ for all $i \in I$, which is a contradiction.  Therefore, $R \subseteq \bigcup\limits_{\scriptstyle i \in I}{f_i}^{-1}(R)$.
\end{proof}

\begin{propo}\label{CP326}
Let $A$ be an ${\rm LM}_{\theta}-$algebra and let ${\L}_{\theta}(A)$ be the ${\rm l}_{\theta}$P--space associated with $A$. Then the following conditions are equivalent for all $a,b \in A$, $a \leqslant b$:  

\begin{itemize}
\item[\rm{(i)}] $\Theta(a,b) = \Theta_{OS}(G)$ is a Boolean congruencia on $A$,
\item[\rm{(ii)}] $G= \bigcup\limits_{\scriptstyle i \in I}\sigma_A(\phi_i b\wedge \overline{\phi}_i a)$ is  a closed subset of $X(A$, 
\item[\rm{(iii)}] there are $i_j\in I$, $1 \leq j \leq n$ such that $G =\sigma_A(\bigvee\limits_{j =1}^{n}(\phi_{i_j}b \wedge \overline{\phi}_{i_j}a))$ and $\sigma_A(b)\setminus\sigma_A(a) \subseteq G$,
\item[\rm{(iv)}] there are $i_j\in I$, $1 \leq j \leq n$ such that $\sigma_A (b) \setminus \sigma_A(b)\subseteq \sigma_A ( \bigvee\limits_{j=1}^{n} (\phi_{i_j} b \wedge 
\overline{\phi}_{i_j}a))$ and  $\Theta_{OS} (G) ={\bf \Theta}([\bigwedge\limits_{j =1}^{n}(\overline{\phi}_{i_j}b \vee \phi_{i_j}a )))$.
\end{itemize}
\end{propo}

\begin{proof}
(i)\, $\Rightarrow$\, (ii): From the  hypothesis and Proposition \ref{C21} we have that 
$G =(\sigma_A(b)\setminus \sigma_A(a))\cup \bigcup\limits_{\scriptstyle i \in I} {f^A_i}^{-1}(\sigma_A(b) \setminus\sigma_A(a))$. On the other hand, since $\Theta_{OS}(G)$ is a Boolean congruence, by Proposition \ref{PAP324} we infer that $\bigcup\limits_{\scriptstyle i \in I } {f^A_i}^{-1}(\sigma_A(b) \setminus \sigma_A(a) )$ is a closed subset of $X(A)$. Hence, by Lemma \ref{L324} we have that $\sigma_A(b) \setminus\sigma_A(a)\subseteq \bigcup
\limits_{\scriptstyle i \in I}{f^A_i}^{-1}(\sigma_A(b)\setminus \sigma_A(a))$. Therefore,  
$G = \bigcup\limits_{\scriptstyle i \in I}{f^A_i}^{-1}(\sigma_A(b) \setminus \sigma_A(a))$ and taking into account that $\sigma_A$ is an isomorphism we conclude the proof. 

\vspace{2mm}

(ii)\, $\Rightarrow$\, (iii):  The hypothesis implies that $G = \bigcup\limits_{\scriptstyle i \in I}{f^A_i}^{-1}(\sigma_A(b) \setminus \sigma_A(a) )$ is a closed subset of $X(A)$ and so, from Lemma \ref{L324} we have that $\sigma_A(b) \setminus \sigma_A(a) \subseteq G$.   
To complete the proof note that $\{{f^A _i}^{-1}(\sigma_A(b) \setminus \sigma_A(a)):  i \in I\}$ is an open covering of $G$. Then, a simple compactness argument shows that 
$G = \bigcup\limits_{j =1}^{ n}{f^A_{i_j}}^{-1}(\sigma_A(b) \setminus \sigma_A(a))$, and since $\sigma_A$ is an isomorphism,  we infer that  $G= \sigma_A(\bigvee\limits_{j =1}^{n} (\phi_{i_j}b \wedge \overline{\phi}_{i_j}a))$.

\vspace{2mm}

(iii)\, $\Rightarrow$\, (ii): From the hypothesis, $G = \bigcup\limits_{j =1}^{n}{f^A_{i_j}} ^{-1}(\sigma_A(b) \setminus \sigma_A(a))$ and $\sigma_A(b)\setminus \sigma_A(a) \subseteq  \bigcup\limits_{j =1}^{n}{f^A_{i_j}}^{-1}(\sigma_A(b) \setminus \sigma_A(a))$. Then, by (lP5) we infer that $\bigcup\limits_{\scriptstyle i \in I}{f^A_i}^{-1}(\sigma_A(b) \setminus \sigma_A(a)) = \bigcup\limits_{j =1}^{n}{f^A_{i_j}}^{-1}( \sigma_A(b) \setminus \sigma_A(a) )$. Therefore, $G = \bigcup\limits_{\scriptstyle i \in I} {f^A_{i}}^{-1}(\sigma_A(b) \setminus \sigma_A(a))$ is a closed subset of $X(A)$ and bearing in mind that $\sigma_A$ is an isomorphismo we conclude that $G=\bigcup\limits_{\scriptstyle i \in I}\sigma_A(\phi_i b\wedge \overline{\phi}_i a)$.

\vspace{2mm}

(ii)\, $\Rightarrow$\, (i): It is easy to check that $G$ is an open and modal subset of $X(A)$. Therefore, $G \in {\cal C \cal O}_{M}(X(A))$. Then, taking into account that ${\cal C \cal O}_{M}(X(A)) \subseteq {\cal O}_{CS}(X(A))$ we have that $\Theta_{OM}(G) = \Theta_{OS}(G)$ and so, by Theorem \ref{T321} we conclude that  $\Theta_{OS}(G)$ is a Boolean congruence on $A$. 
Furthermore, since $G$ is a closed subset of $X(A)$ by Lemma \ref{L324}, $\sigma_A(b) \setminus \sigma_A(a) \subseteq \bigcup\limits_{\scriptstyle i \in I}{f^A_{i}}^{-1}( \sigma_A(b) \setminus \sigma_A(a))$ which impies that $G =(\sigma_A(b) \setminus \sigma_A(a))\cup  \bigcup\limits_{\scriptstyle i \in I}{f^A_{i}}^{-1}(\sigma_A(b) \setminus \sigma_A(a))$. This equality and Proposition \ref{C21} imply that $\Theta(a,b) = \Theta_{OS}(G)$ and so, the proof is complete. 

\vspace{2mm}

(iii)\, $\Leftrightarrow$\, (iv): It follows as a consequence of the fact that $\Theta_{OM}(\sigma_A(\bigvee\limits_{j =1}^{n} (\phi_{i_j}b \wedge \overline{\phi}_{i_j}a)))= 
\Theta_S(X(A) \setminus \sigma_A (\bigvee\limits_{j =1}^{n} (\phi_{i_j}b \wedge \overline {\phi}_{i_j}a))) = \Theta _S(\sigma_A (\bigwedge\limits_{j =1}^{n}(\overline{\phi}_{i_j}b \vee \phi_{i_j}a )))$ and that $\Theta_S (\sigma_A (\bigwedge\limits_{j =1}^{n}(\overline{\phi}_{i_j}b \vee \phi_{i_j}a))) = {\bf \Theta}([\bigwedge\limits_{j =1}^{n}(\overline{\phi}_{i_j}b \vee \phi_{i_j}a )))$.
\end{proof}

\

Finally, we will complete this section establishing a characterization of the Boolean congruences on $n-$valued  
\L ukasiewicz algebras.

\vspace{2mm}

\begin{theo}\label{T421}
Let  $A$ be an  ${\rm LM}_{n}-$algebra and let ${\L}(A)$ be the ${\rm l}_{n}$P--space associated with $A$.
Then the lattice  ${\cal C \cal O}_{M}(X(A))$ is isomorphic to the lattice $Con_{bLk_{n}}(A)$ of the Boolean congruences on $A$ and the  isomorphism is the function  $\Theta_{OM}$  defined as in Theorem {\rm\ref{T321}}.
\end{theo}

\begin{proof} 
It is a direct consequence of Theorem \ref{T321} and  Corollary \ref{C323}. 
\end{proof}

\begin{cor}\label{P410}
Let  $A$ be an  ${\rm LM}_{n}-$algebra. Then Boolean and principal congruences coincide.
\end{cor}

\begin{proof}
It is a direct consequence of Proposition \ref{P312} and Theorem \ref{T421}. 
\end{proof}

\

\noindent A. V. Figallo.\\ 
Departamento de Matem\'atica. Universidad Nacional del Sur,\\ 
8000 Bah\'{\i}a Blanca, Argentina.\\
Instituto de Ciencias B\'asicas. Universidad Nacional de  San Juan.\\
5400 San Juan, Argentina.\\
{\it e-mail: avfigallo@gmail.com}

\vspace{2mm}

\noindent I. Pascual.\\
Instituto de Ciencias B\'asicas. Universidad Nacional de San Juan.\\
5400 San Juan, Argentina.\\
{\it e-mail:  ipascualdiz@gmail.com}

\vspace{2mm}

\noindent A. Ziliani.\\ 
Departamento de Matem\'atica. Universidad Nacional del Sur.\\
8000 Bah\'{\i}a Blanca, Argentina.\\
{\it e-mail: aziliani@gmail.com}

\end{document}